\documentclass[11pt,a4paper]{article}

\usepackage[utf8]{inputenc}
\usepackage[T1]{fontenc}
\usepackage[english]{babel}
\usepackage{amsmath}
\usepackage{amsfonts}
\usepackage{amsthm}
\usepackage{amssymb}
\usepackage{makeidx}
\usepackage{graphicx}
\usepackage{lmodern}
\usepackage[left=2cm,right=2cm,top=2cm,bottom=2cm]{geometry}

\usepackage{color}
\usepackage{comment}
\usepackage[dvipsnames]{xcolor}
\usepackage{graphicx}
\usepackage{framed}
\usepackage{tikz}
\usepackage{calrsfs}
\DeclareMathAlphabet{\pazocal}{OMS}{zplm}{m}{n}
\DeclareMathAlphabet{\mathbfcal}{OMS}{cmsy}{b}{n}

\usepackage{fourier-orns}
\usepackage{mathtools}
\usepackage{relsize}
\usepackage{dsfont}
\usepackage{comment}
\usepackage{hyperref}

\newcommand{\R}{\mathbb{R}}

\newcommand{\Lpazo}{\pazocal{L}}

\newcommand{\Vcal}{\mathcal{V}}

\newcommand{\Ccal}{\mathcal{C}}

\newcommand{\Xcal}{\mathcal{X}}

\newcommand{\Id}{\textnormal{Id}}

\newcommand{\Lip}{\textnormal{Lip}}

\newcommand{\xb}{\boldsymbol{x}}
\newcommand{\yb}{\boldsymbol{y}}

\newcommand{\vb}{\boldsymbol{v}}
\newcommand{\wb}{\boldsymbol{w}}

\newcommand{\Lb}{\boldsymbol{L}}

\newcommand{\Lpazob}{\mathbfcal{L}}
\newcommand{\Lbx}{\boldsymbol{L}_{\xi}}

\newcommand{\Bxi}{\boldsymbol{\xi}}

\newcommand{\Bphi}{\boldsymbol{\phi}}
\newcommand{\BPhi}{\boldsymbol{\Phi}}

\newcommand{\INTSeg}[4]{\int_{#3}^{#4} #1 \textnormal{d} #2}

\newcommand{\Norm}[1]{\parallel \hspace{-0.1cm} #1 \hspace{-0.1cm} \parallel}
\newcommand{\derv}[2]{\frac{\textnormal{d} #1}{ \textnormal{d} #2}}

\newcommand{\textbn}[1]{\textnormal{\textbf{#1}}}

\newtheorem{rmk}{Remark}
\newtheorem{lem}{Lemma}
\newtheorem{Def}{Definition}
\newtheorem{thm}{Theorem}
\newtheorem{prop}{Proposition}

\newtheorem{cor}{Corollary}
\newtheorem{nota}{Notation}

\newtheorem{taggedhypx}{Hypotheses}
\newenvironment{taggedhyp}[1]
    {\renewcommand\thetaggedhypx{#1}\taggedhypx}
    {\endtaggedhypx}

\title{Consensus and Flocking under Communication Failures for a Class of Cucker-Smale Systems}

\author{Beno\^it Bonnet\footnote{Inria Paris $\mathsmaller{\&}$ Laboratoire Jacques-Louis Lions, Sorbonne Université, Université Paris-Diderot SPC, CNRS, Inria, 75005 Paris, France \textit{Email:} \texttt{benoit.a.bonnet@inria.fr}} \,and \'Emilien Flayac\footnote{University of Melbourne, Parkville VIC 3010, Melbourne, Australia. \textit{Email:} \texttt{emilien.flayac@unimelb.edu.au}}}% 

\begin{document}

\maketitle

\begin{abstract}
In this paper, we study sufficient conditions for the emergence of asymptotic consensus and flocking in a certain class of non-linear generalised Cucker-Smale systems subject to multiplicative communication failures. Our approach is based on the combination of strict Lyapunov design together with the formulation of a suitable persistence condition for multi-agent systems. The latter can be interpreted as a lower bound on the algebraic connectivity of the time-average of the interaction graph generated by the communication weights, and provides quantitative decay estimates for the variance functional along the solutions of the system. 
\end{abstract}


\medskip


\section{Introduction}
\label{section:Introduction}

The study of emergent macroscopic structures in dynamical systems describing collective behaviours has been the object of an increasing attention during the past decades, and there is by now a vast and still growing literature devoted to the investigation of asymptotic patterns formation in the class of so-called \textit{cooperative systems} \cite{smith}. These models are indeed widely used in very diverse branches of the field of mathematical modelling, ranging from the study of crowd dynamics \cite{CPT,Piccoli2018}, robot swarms \cite{RobotSwarms} and opinion propagations \cite{AlbiPareschiZanella,HK} to that of animal groups such as bird flocks \cite{ballerini} or fish schools \cite{Albi2014}. 

Since the seminal papers \cite{CS1,CS2}, a great deal of interest has been manifested towards the analysis of the so-called \textit{flocking behaviour} (see Definition \ref{def:Flocking} below) in second-order multi-agent systems (see e.g.  \cite{Caponigro2015,SparseJQMF,ControlKCS} and references therein). The latter describes the appearance of a \textit{consensus} pattern (see Definition \ref{def:consensus} below) in the velocity variable -- otherwise known as \textit{alignment} -- for generalised Cucker-Smale systems of the form 
\begin{equation}
\label{eq:Intro_CS}
\tag{CS}
\left\{
\begin{aligned}
& \dot x_i(t) = v_i(t), \\
& \dot v_i(t) = \frac{1}{N} \sum_{j=1}^N \xi_{ij}(t) \phi(|x_i(t)-x_j(t)|)(v_j(t)-v_i(t)).
\end{aligned}
\right.
\end{equation}
Here $(x_1,\dots,x_N) \in (\R^d)^N$ and $(v_1,\dots,v_N) \in (\R^d)^N$ respectively  stand for the positions and velocities of the agents, while $\phi(\cdot)$ is a positive non-linear kernel which represents the magnitude of their mutual interactions. The functions $\xi_{ij}(\cdot) \in L^{\infty}(\R_+,[0,1])$ are \textit{communication weights}, accounting for potential interaction failures that can occur in the system (e.g. when $\xi_{ij}(t)=0$, see the examples of Section \ref{section:examp} below). Alignment patterns have shown their relevance in many application fields -- in particular for modelling fleets of autonomous vehicles \cite{BeardRen} --, and have been thoroughly investigated in the \textit{full-communication} setting, i.e. when $\xi_{ij}(\cdot) \equiv 1$. When the interactions between agents are subject to possibly severe disruptions, it is then of high interest to identify sufficient conditions under which the formation of asymptotic flocking can still be guaranteed. For discrete-time first- and second-order systems, opinion formation models of this type have been extensively studied in \textit{graph-theoretic} frameworks, see for instance the seminal paper \cite{Moreau2005} and the reference monographs \cite{BeardRen,Bullo2009,Egerstedt2010}. 

Several families of time-varying topologies have already been considered in the context of alignment formation for second-order multi-agents system. In the early work \cite{Tanner2007}, the authors investigated flocking formation for a time-continuous system with non-stationary interaction topologies, under the strong assumption that the corresponding time-dependent graphs were always connected and that the switches must exhibit dwell times. These assumptions were subsequently relaxed in \cite{Martin2014}, at the price of restricting the analysis to discrete-time systems in which the maximal spreading in position between agents is a priori bounded. This led to a rather involved sufficient condition for flocking, in which both the aforementioned bound and the spreading in velocity of the initial state intervene. In \cite{RuLiXue2015}, the authors proved the convergence to flocking for a discrete-time version of \eqref{eq:Intro_CS}, in which the  communication rates $(\xi_{ij}(\cdot))$ are piecewise constant realisations of independent stationary Bernoulli processes. This convergence analysis was later improved in \cite{HeMu2019}, where asymptotic flocking was obtained for the same type of discrete-time systems, with random weights $(\xi_{ij}(\cdot))$ that are neither assumed to be symmetric nor independent. The latter result was building on the recent contribution \cite{Dong2017}, in which flocking formation was studied in the case of discrete-time systems with full communications and \textit{directed} -- i.e. asymmetric -- interaction topologies. We also mention the results of \cite{Dalmao2011} on this topic, where flocking was studied for a discrete-time version of \eqref{eq:Intro_CS} with an asymmetric and slightly more general right-hand side, under the additional structural assumption that the agents are all hierarchically directed towards a common leader (see also \cite{Tang2020}). We stress that such models of communication failures -- expressed in terms of time-varying interaction topologies -- are substantially different from several other known contributions in the literature such as \cite{Ahn2010,Ha2017,Ha2009}. Therein, the agents are assumed to be fully communicating at all times, and the disturbances in the system are modelled by means of additive white noises. In this context, it would seem that the analysis of flocking formation for general time-continuous systems of the form \eqref{eq:Intro_CS} in which the communication weights $(\xi_{ij}(\cdot))$ are merely measurable and do not exhibit any kind of hierarchical structure is still a completely open problem, even in the case of symmetric communications corresponding to \textit{undirected} interaction topologies. 

In order to establish convergence results towards consensus or flocking for general time-continuous non-linear systems, the best identified setting is that of Lyapunov analysis. Indeed in the seminal work \cite{HaLiu}, the authors proposed a simple proof of the emergence of asymptotic flocking for classical Cucker-Smale models, based on the derivation of strict-dissipation inequalities for the velocity variance functional (see Definition \ref{def:Variance} below) along the solutions of the system. However, this methodology cannot be transposed directly to the case where the weights $(\xi_{ij}(\cdot))$ may vanish arbitrarily often, since the variance functional is not strictly dissipative any more in this context. A natural idea to circumvent this difficulty is to try and formulate a suitable \textit{persistence of excitation} condition on the communication weights. Persistence conditions are indeed quite standard in classical control theory \cite{Narendra1989} -- notably for designing stabilising feedbacks \cite{Chaillet2008,Chitour2010} --, and have proven their adaptability in stability analysis at large by allowing to build \textit{strict} Lyapunov functions for non-stationary perturbations of asymptotically stable systems, see e.g. \cite{Maghenem2017,Maghenem2018} and the reference monograph \cite{MazencMalisoff}. In addition to their practical interest, strict Lyapunov functions provide quantitative convergence properties towards the equilibrium, which turn out to be crucial to ensure the formation of asymptotic flocking in the context of multi-agent systems analysis, as amply discussed below (see also \cite{Ha2009,HaLiu}). 

\medskip

The contributions of this article are twofold. The first one lies in the formulation of a suitable persistence condition for systems of the form \eqref{eq:Intro_CS}, that is adapted to the study of alignment patterns in second-order systems. We shall say that a collection of weights $(\xi_{ij}(\cdot))$ satisfies the persistence condition (PE) (see Definition \ref{def:Persistence} below) if there exists a pair $(\tau,\mu) \in \R_+^* \times (0,1]$ such that 
\begin{equation}
\label{eq:Intro_PE}
B \bigg( \Big( \tfrac{1}{\tau} \mathsmaller{\INTSeg{\Lbx(s)}{s}{t}{t+\tau}} \Big) \vb,\vb \bigg) \geq \mu B(\vb,\vb),
\end{equation}
for all $\vb \in (\R^d)^N$. Here, $B : (\R^d)^N \times (\R^d)^N \rightarrow \R$ denotes the \textit{variance bilinear form} (see Definition \ref{def:Variance} below), and $\Lbx(\cdot)$ is the time-dependent graph-Laplacian associated to the interaction weights $(\xi_{ij}(\cdot))$ of the system (see equation \eqref{eq:GraphLaplacian_Def2} below). In the context of cooperative dynamics, the persistence condition proposed in \eqref{eq:Intro_PE} has both a deep and simple meaning in terms of interaction topology. Indeed, it transcribes the fact that on average on any time window of length $\tau > 0$, the communication graph describing the interactions of the agents is connected. It also imposes a uniform lower-bound $\mu \in (0,1]$ on the so-called \textit{algebraic connectivity} of the averaged graph associated to the weights $(\xi_{ij}(\cdot))$ (see Definition \ref{def:AlgebraicConnect} below), which is the first non-zero eigenvalue of the averaged graph-Laplacian. In the way we formulate it, the persistence condition \eqref{eq:Intro_PE} further encodes two interesting ideas. Firstly, it only requires the system to be persistently exciting with respect to the agents which have not reached flocking yet. Secondly, it solely involves the communication weights $(\xi_{ij}(\cdot))$ and not the kernel $\phi(\cdot)$. For this reason, our main result Theorem \ref{thm:Flocking} cannot be recovered as a consequence of earlier contributions e.g. from \cite{Blondel2005,Hendrickx2012,Moreau2005}, where the whole graph-Laplacian is assumed to be persistent. Incidentally for \eqref{eq:Intro_CS}, this would implicitly  boil down to assuming that the maximal distance between agents is a priori bounded. While unharmful in the analysis of first-order consensus systems (see the proof of Theorem \ref{thm:Consensus} below), this is highly problematic when studying alignment formation in second-order systems, as the main difficulty to be handled is precisely that the spreading in position of the agents may diverge. 

Our second contribution is the explicit construction of time-varying trajectory-based Lyapunov functions in the spirit of \cite{MazencMalisoff} for \eqref{eq:Intro_CS}, obtained by combining the variance bilinear form and the persistence conditions \eqref{eq:Intro_PE}. We show that these functionals are strictly dissipative on a family of finite time intervals whose upper-bounds can be chosen to be arbitrarily large, which allows us to recover the non-uniform exponential convergence towards consensus for a first-order variant of \eqref{eq:Intro_CS} (see Theorem \ref{thm:Consensus} below) as well as the non-uniform exponential formation of flocking for the second-order system proper (see Theorem \ref{thm:Flocking} below). While it is known that asymptotic consensus can be recovered in directed first-order systems under mere infinite-time average connectivity assumptions -- e.g. when $\tau = +\infty$ and $\mu$ is non-constant and possibly vanishing at infinity -- (see e.g. \cite{Hendrickx2012,Moreau2005} and other works in the literature), the corresponding convergence results are inherently non-quantitative. Reciprocally in \cite{Manfredi2016}, it is proven that quantitative connectivity conditions in the spirit of \eqref{eq:Intro_PE} are in fact \textit{necessary} for the formation of exponential consensus in first-order multi-agent systems. This fact along with the seminal contributions of \cite{Carrillo2010,HaLiu} suggests that asymptotic flocking formation seems unlikely in the absence of a strictly dissipative structure, supported by some form of quantitative connectivity conditions on the underlying interaction topology. 

Finally, we would like to mention that a wealth of quantitative and non-quantitative persistence-like conditions have already been considered in the multi-agent literature devoted to consensus problems in first-order systems with time-varying interaction topologies (see e.g. \cite{BeardRen,Blondel2005,Manfredi2016,Moreau2005,Hendrickx2012,Tang2020}), as well as to design synchronising controls in second-order robot ensembles (see for instance \cite{Dasdemir2014,Maghenem2019}). However, to the best of our knowledge, this article is the first one to formulate a persistence condition in terms of the positive-definiteness of the averaged graph-Laplacian generated solely by the communication weights with respect to the variance bilinear form. In this regard, it presents the advantage of not incorporating -- either explicitly or implicitly -- any structural assumption on the interaction topology, other than being undirected. Moreover, the study of such a condition to perform a strict Lyapunov design for general time-continuous non-linear alignment systems is also new in the literature.

\medskip

The structure of the article is the following. In Section \ref{section:Consensus}, we introduce our Lyapunov approach by recovering a known result of non-uniform exponential consensus formation for persistently excited first-order dynamics. We then build on these concepts  in Section \ref{section:Flocking} to establish the formation of non-uniform exponential flocking in a class of Cucker-Smale type systems satisfying the strengthened fat tail condition \ref{hyp:K}, which is the main result of this article. In Section \ref{section:examp}, we illustrate our persistence condition on a general class of communication weights, and we conclude with some remarks and open perspectives in Section \ref{section:Conclusion}. 

%%%%%%%%%%%%%%%%%%%%%%%%%%%%%%%%%%%%%%%%%%%%%%%%%%%%%%%%%%%%%%%%%%%%%%%%%%%%%%%%
%								NEW SECTION AHEAD							   %
%%%%%%%%%%%%%%%%%%%%%%%%%%%%%%%%%%%%%%%%%%%%%%%%%%%%%%%%%%%%%%%%%%%%%%%%%%%%%%%%

\section{Consensus formation in first-order Cucker-Smale systems}
\label{section:Consensus}

In this section, we introduce the main tools used throughout this article in the particular case of consensus formation. In this context, we study first-order cooperative systems of the form
\begin{equation}
\label{eq:CS1}
\tag{CS1}
\left\{
\begin{aligned}  
& \dot x_i(t) = \frac{1}{N} \sum\limits_{j=1}^N \xi_{ij}(t) \phi(|x_i(t) - x_j(t)|)(x_j(t) - x_i(t)), \\
& x_i(0) = x_i^0,
\end{aligned}
\right.
\end{equation}
where $(x_1^0,\dots,x_N^0) \in (\R^d)^N$ is a given initial datum. We assume that $\phi \in \Lip(\R_+,\R_+^*)$ where $\R_+^* := \R_+ \backslash \{ 0\}$ denotes the set of positive real numbers, and that the communication weights $\xi_{ij}(\cdot) \in L^{\infty}(\R_+,[0,1])$ are symmetric, i.e. $\xi_{ij}(t) = \xi_{ji}(t)$ for almost every $t \geq 0$ and any $i,j \in \{1,\dots,N\}$. Below, the notation $\xb = (x_1,\dots,x_N)$ will systematically refer to the total state of the system in $(\R^d)^N$, and we shall denote by $\bar{\xb} = \tfrac{1}{N} \mathsmaller{\sum}_{i=1}^N x_i \in \R^d$ its mean value. 

In what follows, we investigate the formation of \textit{consensus} for systems of the form \eqref{eq:CS1}.

\begin{Def}
\label{def:consensus}
A solution $\xb(\cdot)$ of \eqref{eq:CS1} \textnormal{converges to consensus} if for any $i \in \{1,\dots,N\}$, it holds
\begin{equation*}
\lim\limits_{t \rightarrow +\infty} |x_i(t) - \bar{\xb}(t)| = 0. 
\end{equation*}
\end{Def} 
It is a standard strategy in multi-agent systems analysis to rewrite the equations of \eqref{eq:CS1} over $(\R^d)^N$ in matrix form, as 
\begin{equation}
\label{eq:CSM1}
\tag{$\textnormal{CSM}_1$}
\dot \xb(t) = -\Lpazob(t,\xb(t))\xb(t), \qquad \xb(0) = \xb^0,
\end{equation}
where $\Lpazob : \R_+ \times (\R^d)^N \rightarrow \R^{dN \times dN}$ is the so-called \textit{graph-Laplacian} of the system, defined by 
\begin{equation}
\label{eq:GraphLaplacian_Def1}
(\Lpazob(t,\xb) \yb)_i := \frac{1}{N} \sum\limits_{j=1}^N \xi_{ij}(t) \phi(|x_i-x_j|)(y_i-y_j),
\end{equation}
for almost every $t \geq 0$ and any $\xb,\yb \in (\R^d)^N$. In the sequel, we will also make great use of the \textit{partial graph-Laplacian} $\Lbx : \R_+ \rightarrow \R^{dN \times dN}$ associated to the weights $(\xi_{ij}(\cdot))$, defined by
\begin{equation}
\label{eq:GraphLaplacian_Def2}
(\Lbx(t) \yb)_i := \frac{1}{N} \sum\limits_{j=1}^N \xi_{ij}(t) (y_i-y_j),
\end{equation}
for almost every $t \geq 0$ and any $\yb \in (\R^d)^N$. This reformulation of multi-agent dynamics in terms of semilinear ODEs in the space of configurations is fairly general, and allows for a comprehensive study of both consensus and flocking problems via Lyapunov methods. With this goal in mind, we introduce below the so-called \textit{variance bilinear form}, defined in the spirit of \cite{Caponigro2013,Caponigro2015}. 

\begin{Def} 
\label{def:Variance}
The \textnormal{variance bilinear form} $B : (\R^d)^N \times (\R^d)^N \rightarrow \R$ is defined by
\begin{equation}
\label{eq:DefVariance}
B(\xb,\yb) := \frac{1}{N} \sum_{i=1}^N \langle x_i,y_i \rangle - \langle \bar{\xb} , \bar{\yb} \rangle,
\end{equation}
for any $\xb,\yb \in (\R^d)^N$. It is symmetric and positive semi-definite.
\end{Def}

It is a classical observation in the analysis of finite-dimensional multi-agent systems that the state space $(\R^d)^N$ can be written as an orthogonal sum of the form $(\R^d)^N := \Ccal \oplus \Ccal^{\perp}$, where
\begin{equation*}
\Ccal := \Big\{ \xb \in (\R^d)^N ~\text{s.t.}~ x_1 = \dots = x_N \Big\}, 
\end{equation*}
is the so-called \textit{consensus manifold}, and $\Ccal^{\perp} := \{ \xb \in (\R^d)^N ~\text{s.t.}~ \bar{\xb} = 0\}$. Denoting by $\xb := \xb_{\Ccal} + \xb_{\perp}$ the corresponding decomposition of an element $\xb \in (\R^d)^N$, it can be easily checked from \eqref{eq:DefVariance} that  
\begin{equation}
\label{eq:VarianceOrthogonal}
B(\xb,\xb) = B(\xb_{\perp},\xb_{\perp}), 
\end{equation}
so that $B(\xb,\xb) = 0$ if and only if $\xb \in \Ccal$. Thus, the evaluation $B(\xb,\xb)$ of the variance bilinear form provides the distance between a given $\xb \in (\R^d)^N$ and the consensus manifold $\Ccal$. In the sequel, we will state our results in terms of the \textit{standard deviation} $X(\cdot)$ of a solution $\xb(\cdot)$ of \eqref{eq:CSM1}, defined by
\begin{equation}
\label{eq:StandardDev}
X(t) := \sqrt{B(\xb(t),\xb(t))},
\end{equation} 
for all times $t \geq 0$. We now list some useful properties of $\Lpazob(\cdot,\cdot)$ and $B(\cdot,\cdot)$. 

\begin{prop}
\label{prop:SemiPositiveL}
It holds that $\Lpazob(t,\xb)\yb \in \Ccal^{\perp}$ for almost every $t \geq 0$ and any $\xb,\yb \in (\R^d)^N$, and the graph-Laplacian $\Lpazob(t,\xb) \in \R^{dN \times dN}$ is symmetric and positive semi-definite with respect to $B(\cdot,\cdot)$. Moreover, the variance bilinear form supports the following Cauchy-Schwarz inequality 
\begin{equation}
\label{eq:CauchyVariance}
\hspace{2cm} B(\xb,\yb) \leq \sqrt{B(\xb,\xb)} \sqrt{B(\yb,\yb)}.
\end{equation} 
\end{prop}

\begin{proof}
By summing over $i \in \{1,\dots,N\}$ the components in \eqref{eq:GraphLaplacian_Def1} and recalling that the communication weights $(\xi_{ij}(\cdot))$ are symmetric, i.e. $\xi_{ij}(\cdot) = \xi_{ji}(\cdot)$ for every $i,j \in \{1,\dots,N\}$, one has
\begin{equation*}
\frac{1}{N} \sum_{i=1}^N \big( \Lpazob(t,\xb)\yb \big)_i = \frac{1}{N^2} \sum_{i,j=1}^N \xi_{ij}(t) \phi(|x_i-x_j|)(y_i-y_j) = 0,
\end{equation*}
which can be equivalently written as $\Lpazob(t,\xb)\yb \in \Ccal^{\perp}$. Similarly, observe that
\begin{equation*}
\begin{aligned}
B(\Lpazob(t,\xb)\yb,\yb) & = \frac{1}{N^2} \sum\limits_{i,j=1}^N \xi_{ij}(t) \phi(|x_i-x_j|)\langle y_i,y_i-y_j \rangle = \frac{1}{2N^2} \sum\limits_{i,j=1}^N \xi_{ij}(t) \phi(|x_i-x_j|)|y_i-y_j|^2 \geq 0,
\end{aligned}
\end{equation*}
so that $\Lpazob(t,\xb)$ is symmetric and positive semi-definite with respect to $B(\cdot,\cdot)$. Considering the decompositions $\xb := \xb_{\Ccal} + \xb_{\perp}$ and $\yb := \yb_{\Ccal} + \yb_{\perp}$ of $\xb,\yb \in (\R^d)^N$, it can finally be checked that 
\begin{equation*}
\begin{aligned}
B(\xb,\yb) = B(\xb_{\perp},\yb_{\perp}) = \frac{1}{N} \sum_{i=1}^N \langle (\xb_{\perp})_i , (\yb_{\perp})_i \rangle & \leq \bigg( \frac{1}{N} \sum_{i=1}^N |(\xb_{\perp})_i|^2 \bigg)^{1/2} \bigg( \frac{1}{N} \sum_{i=1}^N |(\yb_{\perp})_i|^2 \bigg)^{1/2} \\
& = \sqrt{B(\xb_{\perp},\xb_{\perp})} \sqrt{B(\yb_{\perp},\yb_{\perp})} = \sqrt{B(\xb,\xb)} \sqrt{B(\yb,\yb)}, 
\end{aligned}
\end{equation*}
where we used \eqref{eq:VarianceOrthogonal} as well as the standard Cauchy-Schwarz inequalities in $\R^d$ and $\R^N$ successively. 
\end{proof}

We are now ready to introduce our notion of \textit{persistence of excitation} for Cucker-Smale type multi-agent dynamics subject to multiplicative communication failures.

\begin{Def} 
\label{def:Persistence}
We say that the weights $(\xi_{ij}(\cdot))$ satisfy the \textnormal{persistence of excitation} condition $(\textnormal{PE})$ if there exists a pair $(\tau,\mu) \in \R_+^* \times (0,1]$ such that
\begin{equation}
\label{eq:PE}
\tag{$\textnormal{PE}$}
B \left( \left(\tfrac{1}{\tau} \mathsmaller{\INTSeg{\Lb_{\xi}(s)}{s}{t}{t+\tau}} \right) \xb,\xb \right) \geq \mu B(\xb,\xb),
\end{equation} 
for almost every $t \geq 0$ and all $\xb \in (\R^d)^N$. 
\end{Def} 

\begin{rmk} Observe that condition \eqref{eq:PE} involves only the communication weights $(\xi_{ij}(\cdot))$ through $\Lbx(\cdot)$, and not the trajectories $\xb(\cdot)$ of the system. Moreover, it is formulated using the bilinear form $B(\cdot,\cdot)$, which encodes the idea that one only needs the persistence to hold along directions which are orthogonal to the consensus manifold $\Ccal$. Finally, \eqref{eq:PE} can be interpreted as a lower bound on the so-called \textnormal{algebraic connectivity} (see e.g. \cite{Martin2014,Egerstedt2010,Motsch2014}) of the average of the interaction graph with weights $(\xi_{ij}(\cdot))$ over every time-window of length $\tau > 0$, as illustrated in Section \ref{section:examp} below. 
\end{rmk}

In the following theorem, we prove that solutions of \eqref{eq:CS1} converge to consensus when the persistence assumption \eqref{eq:PE} holds, with a non-uniform exponential rate. This result is not new in itself, and can be derived from earlier works dealing with consensus formations in undirected graphs, such as \cite{Blondel2005,Moreau2005}. However, the proof strategy that we develop here is original in itself, and treating this familiar case allows for a progressive introduction of the concepts that will be necessary later on in Section \ref{section:Flocking} to prove our main result Theorem \ref{thm:Flocking}. 

\begin{thm}[Non-uniform exponential consensus]
\label{thm:Consensus} 
Let $\phi(\cdot) \in \Lip(\R_+,\R_+^*)$ be a positive kernel and suppose that \eqref{eq:PE} holds with parameters $(\tau,\mu) \in \R_+^* \times (0,1]$. Then for any $\xb^0 \in (\R^d)^N$, there exist constants $\alpha_M,\gamma_M >0$ given by \eqref{eq:ConstantsDef1} such that every solution $\xb(\cdot)$ of \eqref{eq:CS1} starting from $\xb^0$ satisfies
\begin{equation*}
X(t) \leq \alpha_M  X(0) e^{-\gamma_M t}, 
\end{equation*}
for all times $t \geq 0$, with $X(\cdot)$ being defined as in \eqref{eq:StandardDev}. In particular, every solution of \eqref{eq:CS1} converges to consensus with a non-uniform exponential rate. 
\end{thm}

\begin{proof}[Proof of Theorem \ref{thm:Consensus}]
Observe that if $\xb^0 \in \Ccal$, i.e. if the system is initially in a consensus configuration, then $\dot{\xb}(t) = 0$ and $\xb(t) = \xb^0$ for all times $t \geq 0$. Thus, we only need to consider the case $\xb^0 \notin \Ccal$. 

By standard diameter estimates on first-order cooperative systems (see e.g. \cite[Proposition 2.1]{Motsch2014}), there exists a radius $R > 0$ depending only $\xb^0 \in (\R^d)^N$ such that $\max_{i \in \{1,\dots,N\}}|x_i(t)| \leq R$ for all times $t \geq 0$. Therefore, since $\phi(\cdot)$ is positive and continuous, there exist two constants $\gamma_0,\gamma_R > 0$ depending only on $R > 0$ -- and thus on $\xb^0 \in (\R^d)^N$ --, such that 
\begin{equation}
\label{eq:PhiBoundConsensus}
\gamma_0 \leq \min_{r \in [0,2R]} \phi(r) \leq \max_{r \in [0,2R]} \phi(r) \leq \gamma_R.
\end{equation}
Let $\Norm{\Lpazob(t,\xb)}_B$ denote the operator seminorm of $\Lpazob(t,\xb)$ with respect to $B(\cdot,\cdot)$, which is given by 
\begin{equation*}
\begin{aligned}
\Norm{\Lpazob(t,\xb)}_B ~ := \sup_{\yb \in (\R^d)^N} \sqrt{\frac{B \big(\Lpazob(t,\xb) \yb , \Lpazob(t,\xb)\yb \big)}{B(\yb,\yb)}}, 
\end{aligned}
\end{equation*}
and consider the constant 
\begin{equation}
\label{eq:c_def}
c \hspace{0.2cm} := \sup_{(t,\xb)} \bigg\{ \Norm{\Lpazob(t,\xb)}_B^{1/2} ~\text{s.t.}~~ t \geq 0 ~~\text{and}~ \max_{i \in \{1,\dots,N\}} |x_i| \leq R \bigg\},
\end{equation}
which is finite as a consequence of \eqref{eq:PhiBoundConsensus}. We also introduce the time-state dependent family of matrices $\psi_{\tau} : \R_+ \rightarrow \R^{dN \times dN}$, defined by 
\begin{equation}
\label{eq:Psi_Def}
\psi_{\tau}(t) := (1+c^2)\tau \, \Id - \frac{1}{\tau} \INTSeg{\INTSeg{\Lpazob(\sigma,\xb(\sigma))}{\sigma}{t}{s}}{s}{t}{t+\tau},
\end{equation}
where $\Id$ denotes the identity matrix of $(\R^d)^N$. Observe that $\psi_{\tau}(\cdot)$ is Lipschitz continuous and thus differentiable almost everywhere by Rademacher's theorem (see e.g. \cite[Theorem 3.2]{EvansGariepy}), and that its pointwise derivative is given explicitly by
\begin{equation}
\label{eq:PsiDerivative}
\dot \psi_{\tau}(t) = \Lpazob(t,\xb(t)) - \frac{1}{\tau} \INTSeg{\Lpazob(s,\xb(s))}{s}{t}{t+\tau}.
\end{equation}
By definition of $c > 0$ in \eqref{eq:c_def} and the Cauchy-Schwarz inequality \eqref{eq:CauchyVariance} supported by $B(\cdot,\cdot)$, it also holds 
\begin{equation*}
0 \leq B(\Lpazob(t,\xb)\yb,\yb) \leq c^2 B(\yb,\yb),
\end{equation*}
for every $\xb,\yb \in (\R^d)^N$, which by linearity of the integral allows us to derive the estimates 
\begin{equation*}
0 \leq  B \bigg( \Big( \tfrac{1}{\tau} \mathsmaller{\INTSeg{\INTSeg{\Lpazob(\sigma,\xb(\sigma))}{\sigma}{t}{s}}{s}{t}{t+\tau}} \Big) \yb,\yb \bigg) \leq \tau c^2 B(\yb,\yb).
\end{equation*}
These latter provide the following matrix bounds on $\psi_{\tau}(\cdot)$ along solutions of \eqref{eq:CS1}
\begin{equation}
\label{eq:PsiBounds}
\sqrt{\tau} X(t) \leq \sqrt{B(\psi_{\tau}(t)\xb(t),\xb(t))} \leq \sqrt{(1+c^2)\tau} X(t),
\end{equation}
for all $t \geq 0$. This leads us to consider the following trajectory-based candidate Lyapunov function
\begin{equation}
\label{eq:Xcal_Def}
\Xcal_{\tau}(t): = \lambda X(t) + \sqrt{B(\psi_{\tau}(t)\xb(t),\xb(t))},
\end{equation}
where $\lambda > 0$ is a tuning parameter and $\xb(\cdot)$ solves \eqref{eq:CS1}. Notice that by \eqref{eq:PsiBounds}, one also has
\begin{equation}
\label{eq:XcalBound}
(\lambda + \sqrt{\tau})X(t) \leq \Xcal_{\tau}(t) \leq \Big( \lambda + \sqrt{(1+c^2)\tau} \Big) X(t).
\end{equation}
This type of construction is inspired from \cite{MazencMalisoff} and appears quite frequently in the theory of strict Lyapunov design for persistent systems. 

By Proposition \ref{prop:SemiPositiveL}, any solution $\xb(\cdot)$ of \eqref{eq:CS1} satisfies $\bar{\xb}(t) = \bar{\xb}_0$ for all times $t \geq 0$. By the invariance with respect to translations of \eqref{eq:CS1} (see e.g. \cite{MTNS2016}), we can therefore assume without loss of generality that $\bar{\xb}(\cdot) \equiv 0$. Our aim now is to prove that a strictly-dissipative inequality of the form
\begin{equation*}
\dot \Xcal_{\tau}(t) \leq - \gamma \Xcal_{\tau}(t),
\end{equation*}
holds for all times $t \geq 0$, where $\gamma >0$ is a given constant. With this goal in mind, observe first that
\begin{equation*}
\begin{aligned}
\derv{}{t} \sqrt{B(\psi_{\tau}(t) \xb(t),\xb(t))} & = \frac{1}{2\sqrt{B(\psi_{\tau}(t)\xb(t),\xb(t))}} \Big( B(\dot{\psi}_{\tau}(t) \xb(t),\xb(t)) + 2 B(\dot{\xb}(t), \psi_{\tau}(t) \xb(t)) \Big) \\
& = \frac{B \big( \dot \psi_{\tau}(t) \xb(t),\xb(t) \big)}{2 \sqrt{B \big(\xb(t),\psi_{\tau}(t)\xb(t) \big)}} - \frac{B \big( \Lpazob(t,\xb(t))\xb(t),\psi_{\tau}(t)\xb(t) \big)}{ \sqrt{B \big( \psi_{\tau}(t)\xb(t),\xb(t) \big)}}, 
\end{aligned}
\end{equation*}
where we used the facts that $\psi_{\tau}(t)$ is symmetric with respect to the bilinear form $B(\cdot,\cdot)$, and that $B(\psi_{\tau}(t)\xb(t),\xb(t)) > 0$ by \eqref{eq:PsiBounds} since $X(t) > 0$. This in turn allows us to compute the time-derivative
\begin{equation*}
\begin{aligned}
\dot \Xcal_{\tau}(t) & = - \frac{\lambda}{X(t)} B \big(\Lpazob(t,\xb(t))\xb(t),\xb(t) \big) + \frac{B \big( \dot \psi_{\tau}(t) \xb(t),\xb(t) \big)}{2 \sqrt{B \big(\psi_{\tau}(t)\xb(t),\xb(t) \big)}} - \frac{B \big( \Lpazob(t,\xb(t))\xb(t),\psi_{\tau}(t)\xb(t) \big)}{ \sqrt{B \big( \psi_{\tau}(t)\xb(t),\xb(t) \big)}},
\end{aligned}
\end{equation*}
for almost every $t \geq 0$. By using \eqref{eq:PsiDerivative} and \eqref{eq:PsiBounds}, we obtain the following differential estimate
\begin{equation}
\label{eq:ConsensusEstimate1}
\begin{aligned}
\dot \Xcal_{\tau}(t) \leq ~ & - \, \frac{1}{2\sqrt{(1+c^2)\tau}X(t)} B \bigg( \Big( \tfrac{1}{\tau} \mathsmaller{\INTSeg{\Lpazob(s,\xb(s))}{s}{t}{t+\tau}} \Big) \xb(t),\xb(t) \bigg) \\
& + \frac{1}{X(t)} \Big( \frac{1}{2\sqrt{\tau}} - \sqrt{(1+c^2)\tau} - \lambda \Big) B \Big( \Lpazob(t,\xb(t))\xb(t),\xb(t) \Big) \\
& + \frac{1}{\sqrt{B(\psi_{\tau}(t)\xb(t),\xb(t))}} B \bigg( \Big( \tfrac{1}{\tau} \mathsmaller{\INTSeg{\INTSeg{\Lpazob(\sigma,\xb(\sigma))}{\sigma}{t}{s}}{s}{t}{t+\tau}} \Big) \xb(t) , \Lpazob(t,\xb(t))\xb(t) \bigg).
\end{aligned}
\end{equation}
for almost every $t \geq 0$. We start by estimating the first line in \eqref{eq:ConsensusEstimate1}. Notice that since $|x_i(t)| \leq R$ for all times $t \geq 0$ and every $i \in \{1,\dots,N\}$, it holds as a consequence of \eqref{eq:PhiBoundConsensus} that
\begin{equation*}
\min_{1 \leq i,j \leq N} \phi(|x_i(t)-x_j(t)|) \geq \min_{r \in [0,2R]} \phi(r) \geq \gamma_0, 
\end{equation*}
for all times $t \geq 0$. By \eqref{eq:GraphLaplacian_Def1}, this in turn implies 
\begin{equation*}
\begin{aligned}
B \big( \Lpazob(t,\xb(t))\yb,\yb \big) & \geq \frac{1}{2N^2} \sum_{i,j=1}^N \gamma_0 \, \xi_{ij}(t) |y_i-y_j|^2  = \gamma_0 \, B \big( \Lb_{\xi}(t)\yb,\yb \big), 
\end{aligned}
\end{equation*}
for any $\yb \in (\R^d)^N$, where $\Lb_{\xi}(t) \in \R^{dN \times dN}$ is defined in \eqref{eq:GraphLaplacian_Def2} and refers to the graph-Laplacian associated to the communication weights $(\xi_{ij}(t))$ at time $t \geq 0$. This together with \eqref{eq:PE} then yields
\begin{equation}
\label{eq:ConsensusEstimate2}
B \bigg( \Big( \tfrac{1}{\tau} \mathsmaller{\INTSeg{\Lpazob(s,\xb(s))}{s}{t}{t+\tau}} \Big) \xb(t),\xb(t) \bigg) \geq \gamma_0 \, B \bigg( \Big( \tfrac{1}{\tau} \mathsmaller{\INTSeg{\Lb_{\xi}(s)}{s}{t}{t+\tau}} \Big) \xb(t),\xb(t) \bigg) \geq  \mu \gamma_0 X^2(t),
\end{equation}
for all times $t \geq 0$. For the third line of \eqref{eq:ConsensusEstimate1}, one has by definition of the operator norm $\Norm{\cdot}_B$ that
\begin{equation}
\label{eq:ConsensusEstimate3}
\begin{aligned}
& B \bigg( \tfrac{1}{\tau} \Big( \mathsmaller{\INTSeg{\INTSeg{\Lpazob(\sigma,\xb(\sigma))}{\sigma}{t}{s}}{s}{t}{t+\tau}} \Big) \xb(t) , \Lpazob(t,\xb(t))\xb(t) \bigg)  \\
& \hspace{1cm} \leq \sqrt{B \bigg( \Big( \tfrac{1}{\tau} \mathsmaller{\INTSeg{\INTSeg{\Lpazo(\sigma,\xb(\sigma))}{\sigma}{s}{t}}{s}{t}{t+\tau}} \Big) \xb(t) ,\Big( \tfrac{1}{\tau} \mathsmaller{\INTSeg{\INTSeg{\Lpazo(\sigma,\xb(\sigma))}{\sigma}{s}{t}}{s}{t}{t+\tau}} \Big) \xb(t) \bigg)} \\
& \hspace{1.5cm}  \times \sqrt{B \Big( \Lpazo(t,\xb(t)) \xb(t),\Lpazo(t,\xb(t)) \xb(t) \Big)} \\
& \hspace{1cm} \leq \tau c^2 X(t) \sqrt{B \Big(\Lpazob(t,\xb(t))\xb(t), \Lpazob(t,\xb(t))\xb(t) \Big)} \\
& \hspace{1cm} \leq \tau c^2 X(t) \Norm{\Lpazob(t,\xb(t))^{1/2}}_B  \sqrt{B \Big( \Lpazob(t,\xb(t))\xb(t),\xb(t) \Big)} \\
& \hspace{1cm} \leq \tau c^3 \Big( \tfrac{\epsilon}{2}X(t)^2 + \tfrac{1}{2\epsilon} B \big( \Lpazob(t,\xb(t))\xb(t),\xb(t) \big) \Big), 
\end{aligned}
\end{equation}
for any $\epsilon > 0$, where we used the Cauchy-Schwartz inequality \eqref{eq:CauchyVariance} supported by $B(\cdot,\cdot)$ as well as Jensen's and Young's inequalities. Merging \eqref{eq:ConsensusEstimate1},\eqref{eq:ConsensusEstimate2},\eqref{eq:ConsensusEstimate3} and using the estimates of \eqref{eq:PsiBounds}, we obtain
\begin{equation*}
\dot \Xcal_{\tau}(t) \leq  -\bigg( \frac{\mu \gamma_0}{2\sqrt{(1+c^2)\tau}} - \frac{c^3 \sqrt{\tau}}{2} \epsilon \bigg) X(t) + \frac{1}{X(t)} \bigg( \frac{1}{2 \sqrt{\tau}} + \frac{c^3 \sqrt{\tau}}{2\epsilon} - \lambda \bigg) B \big( \Lpazob(t,\xb(t))\xb(t),\xb(t) \big).
\end{equation*}
for any given $\lambda,\epsilon > 0$. Therefore, choosing the parameters
\begin{equation}
\label{eq:ExpLambda}
\epsilon := \frac{\mu \gamma_0}{2 c^3 \tau \sqrt{(1+c^2)}} \qquad \text{and} \qquad \lambda := \frac{1}{2 \sqrt{\tau}} + \frac{c^3 \sqrt{\tau}}{2\epsilon},
\end{equation}
and using \eqref{eq:XcalBound} while recalling that $\Lpazob(t,\xb(t))$ is positive semi-definite with respect to $B(\cdot,\cdot)$, we recover 
\begin{equation*}
\begin{aligned}
\dot \Xcal_{\tau}(t) & \leq - \frac{\mu \gamma_0}{4 \sqrt{(1+c^2)\tau}} X(t) \\
& \leq - \frac{\mu \gamma_0}{4 \sqrt{(1+c^2)\tau} \big(\lambda + \sqrt{(1+c^2)\tau} \big)} \Xcal_{\tau}(t).
\end{aligned}
\end{equation*}
We can then conclude by applying Gr\"onwall's Lemma together with the estimates of \eqref{eq:PsiBounds}, which yields
\begin{equation*}
X(t) \leq \alpha_M X(0) e^{-\gamma_{M} t}, 
\end{equation*}
for all times $t \geq 0$, where $\alpha_M,\gamma_{M} >0$ are given respectively by 
\begin{equation}
\label{eq:ConstantsDef1}
\alpha_M := \Big( \tfrac{\lambda + \sqrt{(1+c^2)\tau}}{\lambda + \sqrt{\tau}} \Big)  \qquad \text{and} \qquad \gamma_M := \frac{\mu \gamma_0}{4 \sqrt{(1+c^2)\tau} \big(\lambda + \sqrt{(1+c^2)\tau} \big)}, 
\end{equation}
with $\gamma_0 >0$ satisfying \eqref{eq:PhiBoundConsensus}, $c>0$ given by \eqref{eq:c_def} and $\lambda > 0$ taken as in \eqref{eq:ExpLambda}. By definition \eqref{eq:StandardDev} of the quantity $X(\cdot)$, we conclude that $\xb(\cdot)$ converges to consensus with a non-uniform exponential rate.
\end{proof}

%%%%%%%%%%%%%%%%%%%%%%%%%%%%%%%%%%%%%%%%%%%%%%%%%%%%%%%%%%%%%%%%%%%%%%%%%%%%%%%%
%								NEW SECTION AHEAD							   %
%%%%%%%%%%%%%%%%%%%%%%%%%%%%%%%%%%%%%%%%%%%%%%%%%%%%%%%%%%%%%%%%%%%%%%%%%%%%%%%%

\section{Flocking formation in second-order systems with strong fat tails}
\label{section:Flocking}

In this section we prove the main result of this article, which is the formation of asymptotic flocking in the following class of Cucker-Smale systems subject to multiplicative communication failures
\begin{equation}
\tag{CS2}
\label{eq:CS2} 
\left\{
\begin{aligned}
\dot x_i(t) & = v_i(t), ~~ & x_i(0) = x_i^0, \\ 
\dot v_i(t) & = \frac{1}{N} \sum\limits_{j=1}^N \xi_{ij}(t) \phi(|x_i(t)-x_j(t)|)(v_j(t)-v_i(t)),~~ & v_i(0) = v_i^0.
\end{aligned}
\right.
\end{equation}
Similarly to Section \ref{section:Consensus}, the dynamics in \eqref{eq:CS2} can be written as the following semilinear evolution 
\begin{equation*}
\label{eq:CSM2}
\tag{$\textnormal{CSM}_2$} 
\begin{cases}
\dot \xb(t) \hspace{-0.2cm} & = \vb(t), \hspace{2.125cm} \xb(0) = \xb^0, \\
\dot \vb(t) \hspace{-0.2cm} & = - \Lpazob(t,\xb(t)) \vb(t), ~~ \vb(0) = \vb^0,
\end{cases}
\end{equation*}
in $(\R^d)^N \times (\R^d)^N$. We now recall the definition of \textit{flocking} formation for solutions of \eqref{eq:CS2}.

\begin{Def}
\label{def:Flocking}
A solution $(\xb(\cdot),\vb(\cdot))$ of \textnormal{\eqref{eq:CS2}} \textnormal{converges to flocking} if for any $i \in \{1,\dots,N\}$, it holds
\begin{equation*}
\sup_{t \geq 0} |x_i(t) - \bar{\xb}(t)| < +\infty \qquad \text{and} \qquad \lim_{t \rightarrow +\infty} |v_i(t) - \bar{\vb}(t)| = 0.
\end{equation*}  
\end{Def} 

When studying asymptotic flocking formation for \eqref{eq:CS2}, we will assume that the positive interaction kernel $\phi(\cdot) \in \Lip(\R_+,\R_+^*)$ satisfies the following additional \textit{strong fat tail} condition. 

\begin{taggedhyp}{\textbn{(K)}}
\label{hyp:K}
There exist two constants $K,\sigma > 0$ and a parameter $\beta \in (0,\tfrac{1}{2})$ such that 
\begin{equation}
\label{eq:Assumptions_Phi}
\phi(r) \geq \frac{K}{(\sigma + r)^{\beta}},
\end{equation}
for any $r \geq 0$. In particular $\phi \notin L^1(\R_+,\R_+^*)$, and up to replacing $\phi(\cdot)$ by this lower estimate we can assume without of generality that $\phi(\cdot)$ is non-increasing. 
\end{taggedhyp}

\begin{rmk}
Hypothesis \ref{hyp:K} is a strengthened version of the usual fat tail condition which requires that $\phi \notin L^1(\R_+,\R_+^*)$, see e.g. \cite{HaLiu}. In the context of the present article, we impose that the Cucker-Smale exponent $\beta$ be less that $\tfrac{1}{2}$, whereas in the literature the expected critical exponent beyond which unconditional flocking may fail to occur is $\beta = 1$, see the discussion in Section \ref{section:Conclusion} below for more details.
\end{rmk}

\begin{rmk}
When $\phi \in \Lip(\R_+,\R_+^*)$ is bounded from below by a positive constant, flocking always occurs for \eqref{eq:CS2} in the full-communication setting, namely when $\xi_{ij}(\cdot) \equiv 1$ (see e.g. \cite{CS1,HaLiu,ControlKCS}). In the case where the communication weights $(\xi_{ij}(\cdot))$ satisfy \eqref{eq:PE}, this result still holds for \eqref{eq:CS2} and can be recovered as a simple consequence of Theorem \ref{thm:Consensus}. On the other hand for slim-tailed kernels $\phi(\cdot) \in L^1(\R_+,\R_+^*)$, one can easily construct examples of initial conditions $(\xb^0,\vb^0) \in (\R^d)^N \times (\R^d)^N$ for which asymptotic flocking already fails in the full-communication setting (see e.g. \cite{Caponigro2015}). 
\end{rmk}

One can check that for all times $t \geq 0$, any solution $(\xb(\cdot),\vb(\cdot))$ of \eqref{eq:CSM2} satisfies
\begin{equation*}
\dot{\bar{\xb}}(t) = \bar{\vb}(t) \qquad \text{and} \qquad \dot{\bar{\vb}}(t) = 0.
\end{equation*}
By invoking again the invariance properties under translations of multi-agent systems, we can assume without loss generality that $\bar{\xb}(\cdot) = \bar{\vb}(\cdot) \equiv 0$, and introduce as before the standard deviation maps
\begin{equation}
\label{eq:StandardDev2}
X(t): = \sqrt{B(\xb(t),\xb(t))} \qquad \text{and} \qquad  V(t): = \sqrt{B(\vb(t),\vb(t))}, 
\end{equation} 
evaluated along solutions of \eqref{eq:CS2}. As a consequence of the symmetry of the weights $(\xi_{ij}(\cdot))$, the system \eqref{eq:CS2} is \textit{weakly dissipative}, in the sense that 
\begin{equation}
\label{eq:Weak_Dissipation}
\dot X(t) \leq V(t) \qquad \text{and} \qquad \dot V(t) \leq 0,  
\end{equation}
for almost every $t \geq 0$. In the seminal paper \cite{HaLiu}, the authors proposed a concise proof of the Cucker-Smale flocking, based on a system of \textit{strictly dissipative inequalities}. More precisely, they showed that as a consequence of the semilinear inequalities
\begin{equation}
\label{eq:HaLiu_Dissipation}
\dot X(t) \leq V(t) \qquad \text{and} \qquad \dot V(t) \leq -\phi(2 \sqrt{N} X(t)) V(t),
\end{equation}
every solution of \eqref{eq:CS2} with $\phi \notin L^1(\R_+,\R_+^*)$ and $\xi_{ij}(\cdot) \equiv 1$ converges to flocking. Our aim is to adapt their strategy while using the persistence condition \eqref{eq:PE} to build a strict Lyapunov function for $\eqref{eq:CS2}$. This is the object of the following theorem, which is the main result of this article.

\begin{thm}[Main result -- Non-uniform exponential flocking]
\label{thm:Flocking}
Let $\phi \in \Lip(\R_+,\R_+^*)$ be a non-increasing kernel satisfying hypothesis \ref{hyp:K} and suppose that \eqref{eq:PE} holds with parameters $(\tau,\mu) \in \R_+^* \times (0,1]$. Then for any $(\xb^0,\vb^0) \in (\R^d)^N \times (\R^d)^N$, there exist a radius $\bar{X}_M > 0$ and constants $\alpha_M,\gamma_M >0$ given by \eqref{eq:ConstantDef2} such that every solution $(\xb(\cdot),\vb(\cdot))$ of \eqref{eq:CS2} starting from $(\xb^0,\vb^0)$ satisfies 
\begin{equation*}
X(t) \leq \bar{X}_M \qquad \text{and} \qquad V(t) \leq \alpha_M V(0) e^{-\gamma_M t}, 
\end{equation*}
for all times $t \geq 0$, with $X(\cdot),V(\cdot)$ being defined as in \eqref{eq:StandardDev2}. In particular, every solution of \eqref{eq:CS2} converges to flocking with a non-uniform exponential decay in the velocity variable. 
\end{thm}

The proof of this result relies on the construction of strict and trajectory-dependent Lyapunov functions for \eqref{eq:CS2}, for which a system of inequalities akin to \eqref{eq:HaLiu_Dissipation} holds on bounded time intervals. This local-in-time dissipation allows us to recover a uniform upper-bound on the standard deviation in position $X(\cdot)$ by a reparametrisation of the time variable, which can in turn be leveraged to establish the (non-uniform) exponential decay of $V(\cdot)$, by repeating the arguments explored in Section \ref{section:Consensus}.

\smallskip

\begin{nota} 
In what follows, we will use the \textnormal{rescaled interaction kernel}, defined by
\begin{equation}
\label{eq:Bphi_Def}
\Bphi_{\tau}(r) := \phi \big( \sqrt{2}N (r + \tau V(0) )\big), 
\end{equation}
for any $r \geq 0$, and denote by $\BPhi_{\tau}(\cdot)$ its uniquely determined primitive which vanishes at $X(0)$, namely
\begin{equation}
\label{eq:BPhi_Def}
\BPhi_{\tau}(X) := \INTSeg{\Bphi_{\tau}(r)}{r}{X(0)}{X}. 
\end{equation}
\end{nota}

The proof of Theorem \ref{thm:Flocking} is split into a series of lemmas, which will progressively highlight the role of the different assumptions made on the system. 

\begin{lem}
\label{lem:PE_Flocking}
Let $(\xb(\cdot),\vb(\cdot))$ be a solution of \eqref{eq:CS2}. If \eqref{eq:PE} holds with $(\tau,\mu) \in \R_+^* \times (0,1]$, then 
\begin{equation}
\label{eq:PE_Flocking}
B \bigg( \Big( \tfrac{1}{\tau} \mathsmaller{\INTSeg{\Lpazob(s,\xb(s))}{s}{t}{t+\tau}} \Big) \wb,\wb \bigg) \geq \mu \, \Bphi_{\tau}(X(t)) B(\wb,\wb),
\end{equation}
for any $\wb \in (\R^d)^N$, where $\Bphi_{\tau}(\cdot)$ is defined as in  \eqref{eq:Bphi_Def}.  
\end{lem}

\begin{proof}
By the definition \eqref{eq:GraphLaplacian_Def1} of $\Lpazob : \R_+ \times (\R^d)^N \rightarrow \R^{dN \times dN}$, one has
\begin{equation}
\label{e-Lemma1}
\begin{aligned}
B \bigg( \Big( \tfrac{1}{\tau} \mathsmaller{\INTSeg{\Lpazob(s,\xb(s))}{s}{t}{t+\tau}} \Big) \wb,\wb \bigg) & = \frac{1}{2 N^2}\sum\limits_{i,j=1}^N \Big( \tfrac{1}{\tau} \mathsmaller{\INTSeg{\xi_{ij}(s) \phi(|x_i(s)-x_j(s)|)}{s}{t}{t+\tau}} \Big)|w_i-w_j|^2 \\
& \geq \frac{1}{2 N^2} \sum\limits_{i,j=1}^N \Big( \tfrac{1}{\tau} \mathsmaller{\INTSeg{\xi_{ij}(s) \phi(\sqrt{2} N X(s))}{s}{t}{t+\tau}} \Big)|w_i-w_j|^2, 
\end{aligned}
\end{equation}
since $\phi(\cdot)$ is non-increasing. As a consequence of the weak dissipation relations \eqref{eq:Weak_Dissipation}, it further holds 
\begin{equation*}
X(s) = X(t) + \INTSeg{\dot X(\sigma)}{\sigma}{t}{s} \leq X(t) + \tau V(0),
\end{equation*}
for all $s \in [t,t+\tau]$. By \eqref{e-Lemma1}, and using again that $\phi(\cdot)$ is non-increasing, we obtain
\begin{equation*}
\begin{aligned}
B \bigg( \Big( \tfrac{1}{\tau} \mathsmaller{\INTSeg{\Lpazob(s,\xb(s))}{s}{t}{t+\tau}} \Big) \wb,\wb \bigg) & \geq \frac{\phi \big(\sqrt{2}N (X(t)+\tau V(0)) \big)}{2 N^2} \sum\limits_{i,j=1}^N \Big( \tfrac{1}{\tau} \mathsmaller{\INTSeg{\xi_{ij}(s)}{s}{t}{t+\tau}} \Big)|w_i-w_j|^2 \\
& = \Bphi_{\tau}(X(t)) B \bigg( \Big( \tfrac{1}{\tau}  \mathsmaller{\INTSeg{\Lb_{\xi}(s)}{s}{t}{t+\tau}} \Big) \wb,\wb \bigg) \\
& \geq \mu \, \Bphi_{\tau}(X(t))  B \left(\wb,\wb \right),
\end{aligned}
\end{equation*}
where we used the definitions of $\Lb_{\xi}(\cdot)$ in \eqref{eq:GraphLaplacian_Def2} and $\Bphi_{\tau}(\cdot)$ in \eqref{eq:Bphi_Def}, as well as \eqref{eq:PE} in the last inequality. 
\end{proof}

Let $\psi_{\tau}(\cdot)$ be defined as in \eqref{eq:Psi_Def}, and consider the candidate Lyapunov function defined by 
\begin{equation}
\label{eq:Vcal_def}
\Vcal_{\tau}(t) := \lambda(t) V(t) + \sqrt{B(\psi_{\tau}(t)\vb(t),\vb(t))}, 
\end{equation}
for all times $t \geq 0$ , where $\lambda(\cdot)$ is a smooth tuning curve. In the following lemma, we establish a first differential decay estimate for $\Vcal_{\tau}(\cdot)$. 
 
\begin{lem}
\label{lem:Flocking_Estimate}
For every real number $\epsilon_0 > 0$, there exists a time horizon $T_{\epsilon_0} := 1/4\epsilon_0^2 > 0$ such that
\begin{equation}
\label{eq:Strong_Dissipation0}
\dot \Vcal_{\tau}(t) \leq - \frac{\mu \, \Bphi_{\tau}(X(t))}{2 \sqrt{(1+c^2)\tau}} V(t).
\end{equation}
for almost every times $t \in [0,2T_{\epsilon_0})$. 
\end{lem}

\begin{proof}
As in Section \ref{section:Consensus}, we can assume without loss of generality that $\vb^0 \notin \Ccal$, the other scenario being trivial. Remark that in our context, the constant $c >0$ defined as in \eqref{eq:c_def} is finite since $\phi(\cdot)$ is positive non-increasing and thus bounded from above over $\R_+$. By adapting the arguments of the proof of Theorem \ref{thm:Consensus} above while using the results of Lemma \ref{lem:PE_Flocking}, we can estimate the time-derivative of $\Vcal_{\tau}(\cdot)$ as 
\begin{equation}
\label{eq:Strong_Dissipation1}
\begin{aligned}
\dot \Vcal_{\tau}(t) \leq ~ &  - \left( \frac{\mu \, \Bphi_{\tau}(X(t))}{2\sqrt{(1+c^2)\tau}} - \frac{c^3 \sqrt{\tau}}{2} \epsilon(t) - \dot \lambda(t)\right) V(t) \\
& + \frac{1}{V(t)} \left( \frac{1}{2 \sqrt{\tau}} + \frac{c^3 \sqrt{\tau}}{2 \epsilon(t)} - \lambda(t) \right) B(\Lpazob(t,\xb(t))\vb(t),\vb(t)).  
\end{aligned}
\end{equation}
The main difference with respect to the analysis conveyed in Section \ref{section:Consensus} lies in the choice of \textit{time-dependent} families of parameters $(\lambda(\cdot),\epsilon(\cdot))$. This modification is needed because $X(\cdot)$ may be unbounded along solutions of \eqref{eq:CS2}, so that the map $t \in \R_+ \mapsto \Bphi_{\tau}(X(t)) \in \R_+^*$ is not uniformly bounded from below by a positive constant any more. 

Given an arbitrary $T_{\epsilon}' > 0$ and a differentiable curve $t \in [0,T_{\epsilon}') \mapsto \epsilon(t) \in \R_+^*$, we define $\lambda(\cdot)$ as
\begin{equation}
\label{eq:Lambda_def}
\lambda(t) := \frac{1}{2 \sqrt{\tau}} + \frac{c^3 \sqrt{\tau}}{2 \epsilon(t)}.
\end{equation}
for all times $t \in [0,T_{\epsilon}')$, which implies in particular that $\dot \lambda(t) = - \tfrac{c^3 \sqrt{\tau}}{2 \epsilon^2(t)} \dot \epsilon(t)$. Let us now choose the curve $\epsilon(\cdot)$ as a solution of the ordinary differential equation
\begin{equation*}
\dot \epsilon(t) = \epsilon^3(t), \qquad \epsilon(0) = \epsilon_0,  
\end{equation*}
for a given constant $\epsilon_0 > 0$. The latter is uniquely determined, and can be written explicitly as 
\begin{equation}
\label{eq:epsilon_def}
\epsilon(t) = \frac{\epsilon_0}{\sqrt{1-2\epsilon_0^2 t}},
\end{equation}
for any $t \in [0,\tfrac{1}{2\epsilon_0^2})$. Plugging the analytical expressions of these curves $(\lambda(\cdot),\epsilon(\cdot))$ in \eqref{eq:Strong_Dissipation1} then yields 
\begin{equation*}
\dot \Vcal_{\tau}(t) \leq - \frac{\mu \, \Bphi_{\tau}(X(t))}{2\sqrt{(1+c^2)\tau}}  V(t), 
\end{equation*}
for almost every $t \in [0,1/2\epsilon_0^2)$, so that \eqref{eq:Strong_Dissipation0} holds with  $T_{\epsilon_0} := 1/4 \epsilon_0^2$. 
\end{proof}

Observe that \eqref{eq:Strong_Dissipation0} involves both the standard deviation $V(\cdot)$ and the Lyapunov functional $\Vcal_{\tau}(\cdot)$. However in order to prove Theorem \ref{thm:Flocking}, we will need estimates which solely involve $V(\cdot)$.

\begin{lem}
\label{lem:Strong_LocalFlocking}
There exists a mapping $\epsilon_0 \in \R_+^* \mapsto X_M(\epsilon_0) \in \R_+$ such that $X(t) \leq X_M(\epsilon_0)$ for all $t \in [0,T_{\epsilon_0}]$. In particular for every $\epsilon_0 >0$, the following local strictly-dissipative inequality holds
\begin{equation}
\label{eq:Final}
V(T_{\epsilon_0}) \leq \left( \tfrac{a_1 + b_1 \epsilon_0}{a_2 + b_2 \epsilon_0} \right) V(0) \exp \left( -\frac{\mu \, \Bphi_{\tau}(X_M(\epsilon_0))}{4(a_3 + b_3 \epsilon_0) \epsilon_0} \right),
\end{equation} 
where $\{a_k,b_k\}_{k=1}^3$ are positive constants which only depend on $(c,\tau)$.   
\end{lem}

\begin{proof}
Choose $\epsilon_0 > 0$ and denote by $(\lambda(\cdot),\epsilon(\cdot))$ the corresponding tuning functions given respectively by \eqref{eq:Lambda_def} and \eqref{eq:epsilon_def}. Similarly to \eqref{eq:PsiBounds}, one has for any solution $(\xb(\cdot),\vb(\cdot))$ of \eqref{eq:CSM2} that 
\begin{equation*}
\sqrt{\tau} V(t) \leq \sqrt{B(\psi_{\tau}(t)\vb(t),\vb(t))} \leq \sqrt{(1+c^2)\tau} V(t),
\end{equation*}
for all times $t \in [0,T_{\epsilon_0}]$. By the definition of $\Vcal_{\tau}(\cdot)$ given in \eqref{eq:Vcal_def} along with our choice of parameter curves $(\lambda(\cdot),\epsilon(\cdot))$, it then holds
\begin{equation*}
\left( \sqrt{\tau} + \tfrac{1}{2 \sqrt{\tau}} + \tfrac{c^3 \sqrt{2\tau}}{4 \epsilon_0} \right)V(t) \leq \Vcal_{\tau}(t) \leq \left( \sqrt{(1+c^2)\tau} +  \tfrac{1}{2 \sqrt{\tau}} + \tfrac{c^3 \sqrt{\tau}}{2 \epsilon_0} \right) V(t),
\end{equation*}
for any $t \in [0,T_{\epsilon_0}]$, where we used the fact that $\epsilon(t) \in [\epsilon_0,\sqrt{2}\epsilon_0]$ on this time interval. By a simple identification of the coefficients, these estimates can be rewritten in the condensed form
\begin{equation}
\label{eq:Vcal_Framing}
\left( \tfrac{a_2}{\epsilon_0} + b_2 \right) V(t) \leq \Vcal_{\tau}(t) \leq \left( \tfrac{a_1}{\epsilon_0} + b_1 \right) V(t), 
\end{equation}
for some constants $\{a_k,b_k\}_{k=1}^2$ depending only on $(c,\tau)$. Now by integrating \eqref{eq:Strong_Dissipation0} on $[0,t]$, we obtain
\begin{equation*}
\begin{aligned}
\Vcal_{\tau}(t) & \leq \Vcal_{\tau}(0)- \frac{\mu}{2 \sqrt{(1+c^2)\tau}} \INTSeg{\Bphi_{\tau}(X(s)) V(s)}{s}{0}{t},
\end{aligned}
\end{equation*}
which together with \eqref{eq:Vcal_Framing} in turn yields 
\begin{equation}
\label{eq:Flocking_Estimate1}
\begin{aligned}
V(t) & \leq \left( \tfrac{a_1 + b_1 \epsilon_0}{a_2 + b_2 \epsilon_0} \right) V(0) - \frac{\mu \epsilon_0}{a_2' + b_2' \epsilon_0} \INTSeg{\Bphi_{\tau}(X(s)) V(s)}{s}{0}{t},
\end{aligned}
\end{equation}
where $(a_2',b_2') := 2 \sqrt{(1+c^2)\tau}(a_2,b_2)$. Recall now that $\dot X(s) \leq V(s)$ by \eqref{eq:Weak_Dissipation}, so that applying the change of variable $r = X(s)$ in \eqref{eq:Flocking_Estimate1}, we recover the integral estimate
\begin{equation}
\label{eq:Flocking_Estimate2}
V(t) \leq \left( \tfrac{a_1 + b_1 \epsilon_0}{a_2 + b_2 \epsilon_0} \right) V(0) - \frac{\mu \epsilon_0}{a_2' + b_2' \epsilon_0} \INTSeg{\Bphi_{\tau}(r)}{r}{X(0)}{X(t)} ~=~ \left( \tfrac{a_1 + b_1 \epsilon_0}{a_2 + b_2 \epsilon_0} \right) V(0) - \frac{\mu \epsilon_0}{a_2' + b_2' \epsilon_0} \BPhi_{\tau}(X(t)), 
\end{equation}
for all times $t \in [0,T_{\epsilon_0}]$.

Since $\Bphi_{\tau} \notin L^1(\R_+,\R_+^*)$, its primitive $\BPhi_{\tau}(\cdot)$ is a strictly increasing map which image continuously spans $\R_+$. It is therefore invertible, and for any $\epsilon_0 > 0$ there exists a radius $X_M(\epsilon_0) >0$ such that
\begin{equation}
\label{eq:Critical_Radius}
X_M(\epsilon_0) = \BPhi_{\tau}^{-1}\left( \frac{2 (a_1 + b_1 \epsilon_0) \sqrt{\big(1+c^2 \big) \tau}}{\mu \epsilon_0}  V(0) \right), 
\end{equation}
or equivalently
\begin{equation}
\label{eq:Critical_RadiusBis}
\frac{\mu \epsilon_0}{a_2' + b_2' \epsilon_0} \BPhi_{\tau}(X_M(\epsilon_0)) = \left( \tfrac{a_1 + b_1 \epsilon_0}{a_2 + b_2 \epsilon_0} \right) V(0).
\end{equation}
Since $\BPhi_{\tau}(\cdot)$ is increasing and $V(\cdot)$ is a non-negative quantity by definition, it necessarily follows by plugging \eqref{eq:Critical_RadiusBis} into \eqref{eq:Flocking_Estimate2} that $X(t) \leq X_M(\epsilon_0)$ on $[0,T_{\epsilon_0}]$. Going back to \eqref{eq:Strong_Dissipation0} combined with \eqref{eq:Vcal_Framing}, we can again use the fact that $\Bphi_{\tau}(\cdot)$ is non-increasing to obtain 
\begin{equation*}
\dot \Vcal_{\tau}(t) \leq - \frac{\mu \epsilon_0 \, \Bphi_{\tau}(X_M(\epsilon_0))}{(a_3 + b_3 \epsilon_0)} \Vcal_{\tau}(t),
\end{equation*}
for almost every $t \in [0,T_{\epsilon_0}]$, where $(a_3,b_3) := \mathsmaller{2 \sqrt{(1+c^2)\tau}}(a_1,b_1)$. By an application of Gr\"onwall's Lemma to $\Vcal_{\tau}(\cdot)$ along with yet another use of \eqref{eq:Vcal_Framing}, we finally recover the decay estimate 
\begin{equation*}
V(T_{\epsilon_0}) \leq \left( \tfrac{a_1 + b_1 \epsilon_0}{a_2 + b_2 \epsilon_0} \right) V(0) \exp \left( -\frac{\mu \, \Bphi_{\tau}(X_M(\epsilon_0))}{4(a_3 + b_3 \epsilon_0) \epsilon_0} \right), 
\end{equation*}
where we used the fact that $T_{\epsilon_0} = 1/4\epsilon_0^2$. 
\end{proof}

Building on the dissipative inequality \eqref{eq:Final} obtained in Lemma \ref{lem:Strong_LocalFlocking}, we can in turn recover an upper-bound on the standard deviation $X(\cdot)$ that is uniform with respect to the parameter $\epsilon_0 > 0$. 

\begin{prop}
\label{prop:Bound}
There exists a uniform radius $\bar{X}_M > 0$ such that $X(t) \leq \bar{X}_M$ for all times $t \geq 0$. 
\end{prop}

\begin{proof}
Using the analytical expression \eqref{eq:Critical_Radius} of $X_M(\epsilon_0)$, we have 
\begin{equation*}
\Bphi_{\tau}(X_M(\epsilon_0)) = \Bphi_{\tau} \circ \BPhi_{\tau}^{-1} \left( A_1 + \frac{A_2}{\epsilon_0} \right), 
\end{equation*}
where $A_1,A_2 >0$ are given constants which depend on $(c,\tau,\mu)$. On the other hand by integrating \eqref{eq:Assumptions_Phi} with respect to $r \in [X(0),X]$ for some $X \geq X(0)$, it also holds 
\begin{equation*}
\Phi \big( X \big) \geq \tfrac{K}{1-\beta} \Big( \big( \sigma + X \big)^{1-\beta} - \big( \sigma + X(0) \big)^{1-\beta}\Big),
\end{equation*}
which can be reformulated as
\begin{equation}
\label{eq:X_Est1}
X \leq \Big( \tfrac{1-\beta}{K} \Phi(X) + (\sigma + X(0))^{1-\beta} \Big)^{\tfrac{1}{1-\beta}} - \sigma, 
\end{equation}
for every $X \geq X(0)$. It can be shown by performing a change of variable in \eqref{eq:BPhi_Def} that $\Phi(X) \leq A_3 \, \BPhi_{\tau}(X) + A_4$ for given constants $A_3,A_4 > 0$ depending only on $(V(0),N,\tau)$, so that choosing $X := X_M(\epsilon_0) = \BPhi_{\tau}^{-1} \big( A_1 + A_2/\epsilon_0 \big)$ and recalling that $\Bphi_{\tau}(\cdot)$ is non-increasing, we obtain as a consequence of \eqref{eq:X_Est1} together with hypothesis \ref{hyp:K} that
\begin{equation}
\label{eq:X_Est2}
\Bphi_{\tau} (X_M(\epsilon_0)) \geq \Bphi_{\tau} \bigg( \Big( C_1 + \tfrac{C_2}{\epsilon_0} \Big)^{\tfrac{1}{1-\beta}} - \sigma \bigg) \geq K \Big( C_1 + \tfrac{C_2}{\epsilon_0} \Big)^{\tfrac{\beta}{\beta-1}},
\end{equation}
where $C_1,C_2> 0$ only depend on $(X(0),V(0),N,\sigma,K,c,\tau,\mu)$. Plugging the expression derived in \eqref{eq:X_Est2} into \eqref{eq:Final} while recalling that $T_{\epsilon_0} = 1/4\epsilon_0^2$, we finally recover 
\begin{equation}
\label{eq:Asymptotic_V}
V(T_{\epsilon_0}) \leq C_3 \exp \bigg( -C_4 \, \mu T_{\epsilon_0}^{\tfrac{1-2\beta}{2(1-\beta)}} \bigg),
\end{equation}
for every $\epsilon_0 > 0$, where $C_3,C_4 > 0$ are constants depending only on $(X(0),V(0),N,\sigma,K,c,\tau,\mu)$. 

Observe now that since $\epsilon_0 > 0$ is a free parameter and $\epsilon_0 \in \R_+^* \mapsto T_{\epsilon_0} \in \R_+^*$ continuously spans the whole of $\R_+^*$, we can define a time reparametrisation using $T := T_{\epsilon_0}$. Then by \eqref{eq:Asymptotic_V}, the weak-dissipativity \eqref{eq:Weak_Dissipation} of \eqref{eq:CSM2} expressed in terms of this new time variable writes
\begin{equation*}
\begin{aligned}
\sup_{T \geq 0} X(T) & \leq X(0) + \INTSeg{V(T)}{T}{0}{+\infty} \\
& \leq X(0) + \INTSeg{C_3 \exp \bigg( -C_4 \, \mu T^{\tfrac{1-2\beta}{2(1-\beta)}} \bigg)}{T}{0}{+\infty} < +\infty,
\end{aligned}
\end{equation*}
as we assumed in \ref{hyp:K} that $\beta \in (0,\tfrac{1}{2})$. Thus, there exists a constant $\bar{X}_M > 0$ such that $X(t) \leq \bar{X}_M$ for all times $t \geq 0$, which concludes the proof of our claim.  
\end{proof}

Building on the uniform estimate derived in Proposition \ref{prop:Bound}, we prove our main result Theorem \ref{thm:Flocking}. 

\begin{proof}[Proof of Theorem \ref{thm:Flocking}]
Since we have shown in Proposition \ref{prop:Bound} that $X(\cdot)$ is uniformly bounded, the non-uniform exponential convergence of $V(\cdot)$ towards 0 can be obtained by simply repeating the arguments developed in Section \ref{section:Consensus} for consensus problems. Indeed because $\phi(\cdot)$ is non-increasing, it holds that $\phi(\sqrt{2}NX(t)) \geq \phi(\sqrt{2}N\bar{X}_M)$ for all times $t \geq 0$. Whence, defining the constants 
\begin{equation*}
\epsilon_M := \frac{\mu \, \phi(\sqrt{2}N\bar{X}_M)}{2 c^3 \tau \sqrt{(1+c^2)}} \qquad \text{and} \qquad \lambda_M := \frac{1}{2\sqrt{\tau}} + \frac{c^3 \sqrt{\tau}}{2 \epsilon_M}, 
\end{equation*}
and repeating the estimates detailed in the proof of Theorem \ref{thm:Consensus} above for the functional 
\begin{equation*}
\Vcal_{\tau,M}(t) := \lambda_M V(t) + \sqrt{B(\psi_{\tau}(t) \vb(t),\vb(t))},
\end{equation*}
with $\psi_{\tau}(\cdot)$ being given as in \eqref{eq:Psi_Def}, we recover for almost all times $t \geq 0$ the uniform decay estimate
\begin{equation}
\label{eq:ExpLambdaM}
\dot{\Vcal}_{\tau,M}(t) \leq - \frac{\mu \, \phi(\sqrt{2}N\bar{X}_M)}{4 \sqrt{(1+c^2)\tau} \big( \lambda_M + \sqrt{(1+c^2)\tau} \big)} \Vcal_{\tau,M}(t). 
\end{equation}
By applying Gr\"onwall's Lemma while observing that for all times $t \geq 0$, it holds
\begin{equation*}
(\lambda_M + \sqrt{\tau}) V(t) \leq \Vcal_{\tau,M}(t) \leq \Big( \lambda_M + \sqrt{(1+c^2)\tau} \Big) V(t)
\end{equation*}
we can finally conclude that
\begin{equation*}
X(t) \leq \bar{X}_M \qquad \text{and} \qquad V(t) \leq \alpha_M V(0) e^{-\gamma_M t}, 
\end{equation*}
for all times $t \geq 0$, where $\alpha_M,\gamma_M > 0$ are given by 
\begin{equation}
\label{eq:ConstantDef2}
\alpha_M := \Big( \tfrac{\lambda_M + \sqrt{(1+c^2)\tau}}{\lambda_M + \sqrt{\tau}} \Big) \qquad \text{and} \qquad \gamma_M := \frac{\mu \phi(\sqrt{2}N \bar{X}_M)}{4 \sqrt{(1+c^2)\tau} \big(\lambda_M + \sqrt{(1+c^2)\tau} \big)}, 
\end{equation}
with $\lambda_M > 0$ as in \eqref{eq:ExpLambdaM}. By definition \eqref{eq:StandardDev2} of $X(\cdot),V(\cdot)$, we conclude that $(\xb(\cdot),\vb(\cdot))$ converges to flocking with a non-uniform exponential rate in the velocity variable.
\end{proof}

%%%%%%%%%%%%%%%%%%%%%%%%%%%%%%%%%%%%%%%%%%%%%%%%%%%%%%%%%%%%%%%%%%%%%%%%%%%%%%%%
%								NEW SECTION AHEAD							   %
%%%%%%%%%%%%%%%%%%%%%%%%%%%%%%%%%%%%%%%%%%%%%%%%%%%%%%%%%%%%%%%%%%%%%%%%%%%%%%%%

\section{Illustration of the persistence condition}
\label{section:examp}

In this section, we exhibit a general situation in which \eqref{eq:PE} holds. We start by fixing a constant $\mu \in (0,1]$ and recalling known facts about graph-Laplacians, for which we refer the reader e.g. to \cite{Egerstedt2010}. 

\begin{Def}
\label{def:AlgebraicConnect}
The \textnormal{algebraic connectivity} of a graph with weights $(\xi_{ij})$ is the smallest non-zero eigenvalue of $\Lbx$ seen as an $N \times N$ matrix, and is denoted by $\lambda_2(\Lbx)$.
\end{Def}

\begin{lem}
If an interaction graph with weights $(\xi_{ij})$ is such that $\lambda_2(\Lbx) \geq \mu$, then  
\begin{equation*}
B \big( \Lbx \vb , \vb \big) \geq \mu B(\vb,\vb), 
\end{equation*}
for any $\vb \in (\R^d)^N$. 
\end{lem}

\begin{proof} 
This follows from the definition of algebraic connectivity, along with the fact that
\begin{equation*}
B \big( \Lbx \, \vb , \vb \big) = \frac{1}{2N^2} \sum\limits_{i,j=1}^N \xi_{ij} |v_i - v_j|^2, 
\end{equation*}
for any $\vb \in (\R^d)^N$, see e.g. \cite[Section 2.2]{Motsch2014} for more details. 
\end{proof}

\begin{lem}
Let $\Lb_{\xi_1},\Lb_{\xi_2}$ be the graph-Laplacians associated to two interaction graphs with weights $(\xi_{ij}^1)$ and $(\xi_{ij}^2)$ respectively. Then 
\begin{equation*}
\Lb_{\Bxi} := \Lb_{\xi_1 + \xi_2} = \Lb_{\xi_1} + \Lb_{\xi_2},
\end{equation*}
is the graph-Laplacian of the \textnormal{union} of the two graphs, which weights are $(\Bxi_{ij}) = (\xi_{ij}^1 + \xi_{ij}^2)$.
\end{lem}

From now on, we fix $\tau \in \R_+^*$, an integer $n \geq 1$, and time-dependent communication rates $(\xi_{ij}(\cdot))$ which are constant on all the time intervals of the form $[\tfrac{m \tau}{n},\tfrac{(m+1)\tau}{n})$ for $m \geq 0$. 

\begin{prop}
Suppose that for all $m \geq 0$, the time-average of the graphs $\big\{\xi_{ij}(\tfrac{m+k}{n}\tau) \big\}_{k=0}^{n-1}$, whose weights are given by 
\begin{equation}
\label{eq:AveragedXi}
\Bxi_{ij}^m := \frac{1}{n} \sum\limits_{k=0}^{n-1} \xi_{ij} \big( \tfrac{m + k}{n} \tau \big),
\end{equation}
for any $i,j \in \{1,\dots,N \}$ is connected with $\lambda_2 \big( \Lb_{\Bxi^m} \big) \geq \mu$. Then \eqref{eq:PE} holds. 
\end{prop}

\begin{proof}
For $m\geq 0$ and $t \in [\tfrac{m\tau}{n},\tfrac{(m+1)\tau}{n})$, we have 
\begin{equation}
\label{eq:IneqGraphLap}
\frac{1}{\tau} \INTSeg{\Lbx(s)}{s}{t}{t+\tau} = \Big( \tfrac{(m+1)}{n} - \tfrac{t}{\tau} \Big) \Lbx( \tfrac{m \tau}{n}) + \frac{1}{n}\sum_{k=1}^{n-1} \Lbx \big( \tfrac{m + k}{n} \tau \big) + \big( \tfrac{t}{\tau} - \tfrac{m}{n} \big)  \Lbx(\tfrac{(m+n)\tau}{n}). 
\end{equation}
Now, remark that $\max \big\{\tfrac{(m+1)}{n} - \tfrac{t}{\tau},\tfrac{t}{\tau} - \tfrac{m}{n} \big\} \geq \tfrac{1}{2n}$. Without loss of generality, assume that $\tfrac{(m+1)}{n} - \tfrac{t}{\tau} \geq \tfrac{1}{2n}$, so that by \eqref{eq:IneqGraphLap} it holds that 
\begin{equation*}
\begin{aligned}
& B \bigg( \Big( \tfrac{1}{\tau} \mathsmaller{\INTSeg{\Lbx(s)}{s}{t}{t+\tau}} \Big) \vb , \vb \bigg) \geq B \bigg( \Big( \tfrac{1}{2n} \mathsmaller{\sum}\limits_{k=0}^{n-1} \Lbx \big( \tfrac{m + k}{n} \tau \big) \Big) \vb , \vb \bigg) = B \Big( \Lb_{\Bxi^m/2} \, \vb , \vb \Big) \geq \tfrac{\mu}{2} B(\vb,\vb),
\end{aligned}
\end{equation*}
for all $\vb \in (\R^d)^N$, where the weights $(\Bxi_{ij}^m)$ are defined as in \eqref{eq:AveragedXi}.
\end{proof}

\begin{cor}
Suppose that the piecewise constant weights $(\xi_{ij}(\cdot))$ take their values in an arbitrary \textnormal{finite} set $I \subset [0,1]$. Then \eqref{eq:PE} holds if and only if for all $m \geq 0$, the time-averaged graph whose weights $(\Bxi_{ij}^m)$ are given by \eqref{eq:AveragedXi} is connected.
\end{cor}

\begin{proof} 
The direct implication of this statement is evident. For the converse one, observe that since $I \subset [0,1]$ is a finite set, there only exists a finite number of graphs with weights given by \eqref{eq:AveragedXi} which are connected. In particular, the quantity
\begin{equation*}
\mu := \min \Big\{ \lambda_2 \big( \Lb_{\Bxi^m} \big) ~\text{s.t. $(\Bxi_{ij}^m)$ are given by \eqref{eq:AveragedXi} and generate a connected graph} \Big\},
\end{equation*}
is positive and independent of $m \geq 0$. Thus, \eqref{eq:PE} holds with parameters $(\tau,\mu) \in \R_+^* \times (0,1]$.  
\end{proof}

We now illustrate these general results for piecewise constant communication rates on a simple example with $N=4$ agents. For $\tau \in \R_+^*$ and $t \geq 0$, consider the interactions weights defined as follows
\begin{equation}
\label{eq:ExampWeight}
\begin{aligned}
\xi_{14}(t) & = 
\left\{
\begin{aligned}
& 1 ~~ \text{if $ \lfloor t/\tau \rfloor = 1 \, \textnormal{mod}[6]$,} \\
& 0 \hspace{1.7cm} \text{otherwise,}
\end{aligned}
\right. 
\qquad \xi_{34}(t) = 
& \left\{
\begin{aligned}
& 1 ~~ \text{if $ \lfloor t/\tau \rfloor = 3 \, \textnormal{mod}[6]$,} \\
& 0 \hspace{1.7cm} \text{otherwise,}
\end{aligned}
\right. \\
\xi_{23}(t) =  \xi_{24}(t) & =
\left\{
\begin{aligned}
& 1 ~~ \text{if $ \lfloor t/\tau \rfloor = 5 \, \textnormal{mod}[6]$,} \\
& 0 \hspace{1.7cm} \text{otherwise,}
\end{aligned}
\right.
\end{aligned}
\end{equation}
where $\lfloor \cdot \rfloor$ denotes the lower integer part of a real number, and set all the other weights to $0$. In this example, our signals are piecewise constant on intervals of the form $[\tfrac{m\tau}{6},\tfrac{(m+1)\tau}{6})$ for any $m \geq 0$.

\begin{figure}[!ht]
\centering
\resizebox{0.6 \linewidth}{!}{
\begin{tikzpicture}
% First configuration
\draw (0,0) node[shape=circle, draw] {$1$}; 
\draw (2,0) node[shape=circle, draw] {$2$};
\draw (2,-2) node[shape=circle, draw] {$3$};
\draw (0,-2) node[shape=circle, draw] {$4$};
\draw[<->, blue, line width = 0.35mm] (0,-0.4)--(0,-1.6); 
\draw[blue] (-0.25,-1)node {$1$}; 
\draw[blue] (1,-3.25) node {\large $\lfloor t/\tau \rfloor = 1 \, \textnormal{mod}[6]$};
% 
% Second configuration
\draw[shift={(4 cm,0 cm)}] (0,0) node[shape=circle, draw] {$1$}; 
\draw[shift={(4 cm,0 cm)}] (2,0) node[shape=circle, draw] {$2$};
\draw[shift={(4 cm,0 cm)}] (2,-2) node[shape=circle, draw] {$3$};
\draw[shift={(4 cm,0 cm)}] (0,-2) node[shape=circle, draw] {$4$};
\draw[shift={(4 cm,0 cm)},<->, blue, line width = 0.35mm] (0.4,-2)--(1.6,-2); 
\draw[shift={(4 cm,0 cm)}, blue] (1,-2.3)node {$1$}; 
\draw[shift={(4 cm,0 cm)}, blue] (1,-3.25) node {\large  $\lfloor t/\tau \rfloor = 3 \, \textnormal{mod}[6]$}; 
% 
% Third configuration
\draw[shift={(8 cm,0 cm)}] (0,0) node[shape=circle, draw] {$1$}; 
\draw[shift={(8 cm,0 cm)}] (2,0) node[shape=circle, draw] {$2$};
\draw[shift={(8 cm,0 cm)}] (2,-2) node[shape=circle, draw] {$3$};
\draw[shift={(8 cm,0 cm)}] (0,-2) node[shape=circle, draw] {$4$};
\draw[shift={(10 cm,0 cm)}, <->, blue, line width = 0.35mm] (0,-0.4)--(0,-1.6); 
\draw[shift={(8.95 cm,-2.375 cm)}, rotate = -45,<->, blue, line width = 0.3mm] (-1,0)--(-1,2); 
\draw[shift={(8 cm,0 cm)}, blue] (1,-3.25) node {\large $\lfloor t/\tau \rfloor = 5 \, \textnormal{mod}[6]$}; 
\draw[shift = {(10.5 cm, 0cm)}, blue] (-0.25,-1)node {$1$};
\draw[shift = {(8 cm, 0cm)}, blue] (0.7,-0.8)node {$1$};
\end{tikzpicture}}
\caption{\textit{Illustration of the admissible connections between agents}}
\end{figure}
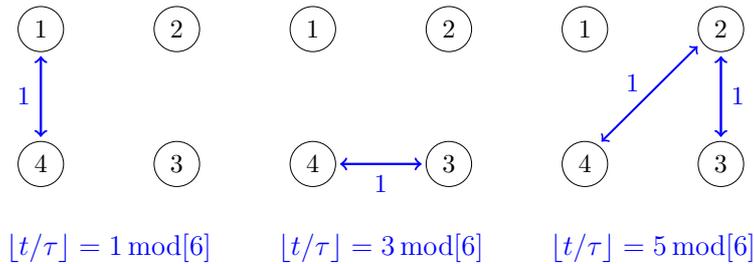

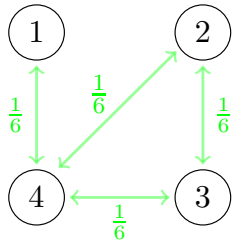
\begin{figure}[!ht]
\centering
\resizebox{0.2 \linewidth}{!}{
\begin{tikzpicture}
\draw (0,0) node[shape=circle, draw] {$1$}; 
\draw (2,0) node[shape=circle, draw] {$2$};
\draw (2,-2) node[shape=circle, draw] {$3$};
\draw (0,-2) node[shape=circle, draw] {$4$};
\draw[<->, opacity =0.4, green, line width = 0.35mm] (0,-0.4)--(0,-1.6); 
\draw[<->, opacity =0.4, green, line width = 0.35mm] (0.4,-2)--(1.6,-2); 
\draw[shift = {(2 cm,0 cm)}, <->, opacity =0.4, green, line width = 0.35mm] (0,-0.4)--(0,-1.6); 
\draw[shift={(0.975 cm,-2.35 cm)}, rotate = -45,<->, <->, opacity =0.4, green, line width = 0.35mm] (-1,0)--(-1,2); 
\draw[green] (-0.25,-1) node {\small $\tfrac{1}{6}$};
\draw[shift = {(2.5cm,0cm)}, green] (-0.25,-1) node {\small $\tfrac{1}{6}$}; 
\draw[green] (1,-2.3) node {\small $\tfrac{1}{6}$}; 
\draw[green] (0.75,-0.75)node {$\tfrac{1}{6}$};
\end{tikzpicture}}
\caption{\textit{Illustration of the averaged interaction graph on a time window of the form $[t,t+\tau]$}}
\end{figure}

Then, the weights $(\xi_{ij}(\cdot))$ defined in \eqref{eq:ExampWeight} are such that the persistence condition \eqref{eq:PE} holds. This can be verified e.g. by computing the smallest positive eigenvalue of the averaged graph-Laplacian matrix $\Lb_{\Bxi^m}$, where $\Bxi^m$ is defined as in \eqref{eq:AveragedXi} with $t \in [\tfrac{m \tau}{6},\tfrac{(m+1)\tau}{6})$. In this example, the spectrum of $\Lb_{\Bxi^m}$ for all $m \geq 0$ is given explicitly by
\begin{equation*}
\textnormal{Sp} \big( \Lb_{\Bxi^m} \big) = \Big\{ 0,\tfrac{1}{6},\tfrac{1}{2},\tfrac{2}{3} \Big\}, 
\end{equation*}
so that \eqref{eq:PE} holds with $\tau \in \R_+^*$ and $\mu := \lambda_2 \big( \Lb_{\Bxi^m} \big) = \tfrac{1}{6}$.

%%%%%%%%%%%%%%%%%%%%%%%%%%%%%%%%%%%%%%%%%%%%%%%%%%%%%%%%%%%%%%%%%%%%%%%%%%%%%%%%
%								NEW SECTION AHEAD							   %
%%%%%%%%%%%%%%%%%%%%%%%%%%%%%%%%%%%%%%%%%%%%%%%%%%%%%%%%%%%%%%%%%%%%%%%%%%%%%%%%

\section{Conclusion and perspectives}
\label{section:Conclusion}

In this article, we proved two convergence results for multi-agent systems subject to general multiplicative communication failures. If the communication rates satisfy a persistence of excitation condition, then one has non-uniform exponential convergence to consensus for first-order systems (Theorem \ref{thm:Consensus}) and to flocking for Cucker-Smale systems, under a strengthened version of the usual fat tail condition on the kernel (Theorem \ref{thm:Flocking}). For the sake of conciseness and readability, we assumed that the initial time of the non-stationary dynamics was fixed and equal to 0. Yet, it could be checked by repeating our argument that both convergence results are \textit{uniform with respect to the initial time}, as the estimates derived on the Lyapunov functionals $\Xcal_{\tau}(\cdot)$ and $\Vcal_{\tau}(\cdot)$ do not exhibit any explicit time-dependence. In the future, we aim at improving our main result Theorem \ref{thm:Flocking} in three directions. 

First, we will investigate whether the rather surprising exponent range $\beta \in (0,\tfrac{1}{2})$ -- which is currently needed in order to ensure that asymptotic flocking occurs -- has an intrinsic meaning, or if it just arises as a limit of our current choice of Lyapunov function. Answering this question might also pave the way for flocking results with weaker interactions, involving confinement conditions linking the initial state and velocity mean-deviations as well as the persistence parameters. 

Then, we will study communication failures defined as the realisations of stochastic processes and try to see under which assumptions and in what sense the convergence towards consensus and flocking can occur (almost surely, in probability, etc...). In this setting, one of the main difficulties will most likely lie in the identification of proper stochastic generalisations of \eqref{eq:PE}. 

Lastly, we will investigate whether our dissipative approach applied here to the standard deviations -- which are $L^2$-functionals --, can be adapted to $L^{\infty}$-type Lyapunov functionals in the spirit of \cite{HaLiu,ControlKCS}. The motivation behind this line of study is that $L^2$-type functionals do not allow for the study of flocking formation in the macroscopic setting as the number $N$ of agents goes to infinity, while $L^{\infty}$-type functionals typically do.  

\begin{flushleft}
{\small \textbf{Acknowledgments}: The authors were partially supported by the Archim\`ede Labex (ANR-11-LABX-0033) and the A*MIDEX project (ANR-11-IDEX-0001-02), funded by the ``Investissements d'Avenir'' French Government program. The authors wish to dearly thank Francesco Rossi for suggesting a preliminary version of the problem that we treated here, as well as Ioannis Sarras for his numerous insights on the topic of strict Lyapunov design for persistent systems.
}
\end{flushleft}

\bibliographystyle{plain}
{\footnotesize
\bibliography{../ControlWassersteinBib}
}

\end{document}